\documentclass[11pt,amstex]{article}
\usepackage{amsmath}
\usepackage{amssymb}
\usepackage{amsfonts}
\usepackage{amscd}
\usepackage[usenames]{color}
\usepackage{setspace}
\usepackage{verbatim}
\usepackage{stmaryrd}
\usepackage{amsthm}
\usepackage{subeqnarray}

\input{xypic}
\xyoption{all}

\newtheorem{theorem}{Theorem}[section]
\newtheorem{prop}[theorem]{Proposition}

\newtheorem{lemma}[theorem]{Lemma}

\newtheorem{Theorem A}[theorem]{Theorem A}

\newtheorem{remark}[theorem]{Remark}
\newenvironment{rem}{\begin{remark}\rm}{\end{remark}}
\newtheorem{example}[theorem]{Example}

\def\la{\lambda}

\def\Z{{\mathbb Z}}

\def\P{{\mathbb P}}
\def\Q{{\mathbb Q}}
\def\cx{{\mathbb C}}
\def\L{{\cal L}}

\def\F{{\cal F}}

\def\M{{\cal M}}
\def\N{{\cal N}}
\def\O{{\cal O}}

\def\D{{\cal D}}
\def\Z{{\cal Z}}

\def\wh#1{\widehat{#1}}

\def\CP{{\cal P}}

\newcommand{\tnorm}[1]{%
  \left\vert\kern-0.9pt\left\vert\kern-0.9pt\left\vert #1
    \right\vert\kern-0.9pt\right\vert\kern-0.9pt\right\vert_{1,p}  }

    \newcommand{\twonorm}[1]{%
  \left\vert\kern-0.9pt\left\vert\kern-0.9pt\left\vert #1
    \right\vert\kern-0.9pt\right\vert\kern-0.9pt\right\vert_{1,2}  }

\newcommand{\pnorm}[1]{%
  \left\vert\kern-0.9pt\left\vert\kern-0.9pt\left\vert #1
    \right\vert\kern-0.9pt\right\vert\kern-0.9pt\right\vert_{p}  }

\newcommand{\ttnorm}[2]{%
  \left\vert\kern-0.9pt\left\vert\kern-0.9pt\left\vert #1
    \right\vert\kern-0.9pt\right\vert\kern-0.9pt\right\vert_{1,p; #2}  }

\newcommand{\ppnorm}[2]{%
  \left\vert\kern-0.9pt\left\vert\kern-0.9pt\left\vert #1
    \right\vert\kern-0.9pt\right\vert\kern-0.9pt\right\vert_{p;#2}  }

\def\non{\noindent}

\def\Aut{{\rm Aut}}

\def\wh{\widehat}

\def\tcup{{\textstyle \bigcup}}

\def\tsqcup{{\textstyle \bigsqcup}}
\def\tprod{{\textstyle \prod}}
\def\tsum{{\textstyle \sum}}

\def\C{{\cal C}}
\def\CP{{\cal P}}

\def\X{{\cal X}}

\def\CZ{{\cal Z}}
\def\f{{\bf f}}
\def\p{{\bf p}}

\title{{\bf Degree three spin Hurwitz numbers}}
\author{Junho Lee }

\date{\empty}
\addtocounter{section}{0}
\begin{document}

\maketitle

\begin{abstract}
\medskip
Recently, Gunningham \cite{G} calculated all spin Hurwitz numbers in terms of combinatorics of Sergeev algebra.
In this paper, we use a spin curve degeneration  to obtain a recursion formula for
degree three spin Hurwitz numbers.

\end{abstract}

\vskip.15in
%%%%%%%%%%%%%%%%%%%%%%%%%%%%%%%%%%%%%%%

%%%%%%%%%%%%%%%%%%%%%%%%%%%%%%%%%%%%%%%%%%%%%%%%%%%%%%%%%%%%
%%%%%%%%%%%%%%%%%  Introduction %%%%%%%%%%%%%%%%%%%%%%%%
%%%%%%%%%%%%%%%%%%%%%%%%%%%%%%%%%%%%%%%%%%%%%%%%%%%%%%%%%%%%%%
\setcounter{section}{0}

Let $D$ be a complex curve of genus $h$ and $N$ be a theta characteristic on $D$, i.e. $N^2=K_D$.
The pair $(D,N)$ is called a {\em spin curve} of genus $h$ with parity $p\equiv h^0(N)$ (mod 2).
For $i=1,\cdots,k$, let $m^i = (m^i_1,\cdots,m_{\ell_i}^i)$ be an odd partition of $d>0$, namely
all components $m^i_j$ are odd.
Fix $k$ points $q^1,\cdots,q^k$ in $D$ and
consider
degree $d$  maps $f:C\to D$ from possibly disconnected domains $C$ of Euler characteristic $\chi$
that are ramified only over the fixed points $q^i$ with ramification data $m^i$.
Observe that the Riemann-Hurwitz formula shows
\begin{equation}\label{RHformula}
2d(1-h)-\chi+ \sum_{i=1}^k \big(\ell(m^i)-d)\ =\   0
\end{equation}
where $\ell(m^i)=\ell_i$ is the length of $m^i$.
By the Hurwitz formula,
the twisted line bundle
\begin{equation}\label{TwistedTheta}
L_f\ =\ f^*N\otimes \O\big(\,\sum_{i,j} \tfrac12 (m^i_j-1) x^i_j\,\big)
\end{equation}
is a theta characteristic on $C$ where
$f^{-1}(q^i)=\{x^i_j\}_{1\leq j\leq \ell_i}$ and $f$ has  multiplicity $m^i_j$ at $x^i_j$.
We define the parity $p(f)$ of a map $f$ by
\begin{equation}\label{parity}
p(f)\ \equiv\  h^0(L_f)\ \ \ \ \ \mbox{(mod 2)}.
\end{equation}
Given odd partitions $m^1,\cdots,m^k$ of $d$,
the spin Hurwitz number of genus $h$ and parity $p$ is defined as a (weighted) sum of (ramified) covers $f$
satisfying (\ref{RHformula}) with sign determined by the parity
$p(f)$:
\begin{equation}\label{spinHurwitz}
H_{m^1,\cdots,m^k}^{h,p}\ =\ \sum_f \frac{(-1)^{p(f)}}{|\Aut(f)|}
\end{equation}
Eskin, Okounkov and  Pandharipande \cite{EOP} calculated the genus $h=1$ and odd parity spin Hurwitz numbers  in terms of characters of Sergeev group.
Recently, Gunningham \cite{G} calculated all spin Hurwitz numbers in terms of combinatorics of Sergeev algebra.

\medskip

The trivial partition $(1^d)$ of $d$ is a  partition  whose components are all one.
If $m^k=(1^d)$, then $f$ has no ramification points over the fixed point $q^k$ and hence we have
\begin{equation}\label{trivial}
H^{h,p}_{m^1,\cdots,m^{k-1},(1^d)}\ =\ H^{h,p}_{m^1,\cdots,m^{k-1}}.
\end{equation}
When all partitions $m^i=(1^d)$, denote
the spin Hurwitz numbers (\ref{spinHurwitz}) by $H^{h,p}_d$.
These  are
dimension zero local GW invariants $GT_d^{loc,h,p}$ of spin curve $(D,N)$ that give all dimension zero  GW invariants of K\"{a}hler surfaces
 with a smooth canonical divisor (cf. \cite{KL1}, \cite{KL2}, \cite{LP1}, \cite{MP}).
For notational simplicity, we set $H^{h,p}_{(3)^0}=H^{h,p}_3$ and for $k\geq 1$ write
$$
%H^{h,p}_{(3)^0}\ \ \ \ \ \mbox{and}\ \ \ \ \
H_{(3)^k}^{h,p}
$$
for the spin Hurwitz numbers
$H_{(3),\cdots,(3)}^{h,p}$  with the same $k$  partitions $(3)$.
Since there are two odd partitions $(1^3)$ and (3) of $d=3$, by (\ref{trivial}) it suffices to compute $H_{(3)^k}^{h,p}$ for $k\geq 0$.
The aim of this paper is to use a spin curve degeneration  to obtain  the following recursion formula.

\begin{theorem}\label{Main}
If $h=h_1+h_2$ and $p\equiv p_1+p_2$ (mod 2) then  for $k_1+k_2=k$
\begin{equation}\label{main1}
H_{(3)^k}^{h,p}\ =\
3!\,H_{(3)^{k_1}}^{h_1,p_1}\cdot H_{(3)^{k_2}}^{h_2,p_2}  \ +\
3\,H_{(3)^{k_1+1}}^{h_1,p_1}\cdot H_{(3)^{k_2+1}}^{h_2,p_2}.
\end{equation}
\end{theorem}

\medskip
One can use  Theorem~\ref{Main} and the result of \cite{EOP} to explicitly compute the degree $d=3$  spin Hurwitz numbers.
In Proposition~\ref{computation}, we show that
\begin{equation}\label{E:computation}
H_{(3)^k}^{h,\pm}\ =\ 3^{2h-2}\big[\,(-1)^k2^{k+h- 1}\pm 1\,\big]
\end{equation}
where $+$ and $-$ denote the even and odd parities.
When the degree $d=1,2$, the dimension zero  local GW invariants are given by the formulas
$$
GT_1^{loc,h,\pm}\ =\ \pm 1\ \ \ \ \ \mbox{and}\ \ \ \ \ GT_2^{loc,h,\pm}\ =\ \pm 2^{h-1}
$$
(cf. Lemma~2.6 of \cite{L}). Since $GT_d^{loc,h,p}=H_d^{h,p}$ as mentioned above,  the formula (\ref{E:computation}) shows
\begin{equation*}\label{deg3lcaoGT}
GT_3^{loc,h,\pm}\ =\
%H_{(3)^0}^{h,\pm}\ =\
3^{2h-2}(2^{h- 1}\pm 1).
\end{equation*}
This calculation is, in fact, the  main motivation for the paper.

\medskip
In Section~1, we express the degree $d$ spin Hurwitz numbers (\ref{spinHurwitz}) in terms of relative GW moduli spaces.
We can then apply a degeneration method for a family of curves $\D\to \Delta$ where
the central fiber $D_0$ is a nodal curve and the general fiber $D_\la$ ($\la\ne 0$) is a smooth curve.
Section~2 describes the relative moduli space $\M_0$ of maps  $f$ into the nodal curve $D_0$. In Section~3,
we show that
the union over $\la\in\Delta$ of relative moduli spaces $\M_\la$ of maps into $D_\la$ consists of
connected components $\Z_{m,f}\to \Delta$ containing $f\in \M_0$. Here  $m$ is the ramification data of $f$ over nodes of $D_0$ such that $d-\ell(m)$ is even.

The (ordinary) Hurwitz numbers are  sums of (ramified) maps modulo automorphism without sign.
One can easily obtain a recursion formula for Hurwitz numbers by counting maps in the general fiber of
$\Z_{m,f}\to \Delta$.
For spin Hurwitz numbers, one needs  to calculate parities of maps  induced from a fixed spin structure on the
family of curves $\D$.

The novelty of our approach is to apply a Schiffer variation for the parity calculation.
The space $\Z_{m,f}$ is, in general, not smooth.
In Section~4, we construct a smooth model for $\Z_{m,f}$ by  Schiffer variation. In Section~5, we use the smooth model to  twist the pull-back of the spin structure on  $\D$. When the degree $d=3$, the  partition $m$ is odd, either $(1^3)$ or $(3)$. In this case, a suitable twisting immediately yields a required parity calculation.
We prove Theorem~\ref{Main} in Section~6 and the formula (\ref{E:computation}) in Section~7.

For higher degree $d\geq 4$, the partition $m$ may not to be odd!
A new parity calculation is needed.
In  \cite{LP2},
we generalized the recursion formula (\ref{main1})
for higher degree spin Hurwitz numbers by employing additional geometric analysis approach
for parity calculation.

\vskip 1cm

%%%%%%%%%%%%%%%%%%%%%%%%%%%%%%%%%%%%%%%%%%%%%%%%%%%%%%%%%%%%%%%%
%%%%%%%%%%%%%%%%%%%%%%%%%%%%%%%%%%%%%%%%%%%%%%%%%%%%%%%%%%%%%%%%
\setcounter{equation}{0}
\section{Dimension zero relative GW moduli spaces  }
\label{section1}
\bigskip

In this section, we express the spin Hurwitz numbers (\ref{spinHurwitz}) in terms of dimension zero
relative GW moduli spaces. We will follow the definitions of \cite{IP2} for the relative GW theory.

\medskip
Let $D$ be a smooth curve of genus $h$ and let  $V=\{q^1,\cdots,q^k\}$ be a fixed set of points on $D$.
Given partitions $m^1,\cdots,m^k$ of $d$,
a degree $d$ holomorphic map $f:C\to D$ from a  possibly disconnected  curve $C$ is called
{\em $V$-regular} with  contact vectors $m^1,\cdots,m^k$ if $f^{-1}(V)$ consists of $\sum \ell(m^i)$ contact  marked points $x^i_j$ ($1\leq j\leq \ell(m^i)$)
with $f(x^i_j)=q^i$ such that  $f$ has  ramification index (or multiplicity)  $m^i_j$ at $x^i_j$.
Two $V$-regular maps $(f,C;\{x^i_j\})$ and $(\tilde{f},\tilde{C};\{\tilde{x}^i_j\})$
are equivalent  if  they are isomorphic, i.e.,
there is a biholomorphism $\sigma:C\to \tilde{C}$ with $\tilde{f}\circ \sigma=f$ and $\sigma(x^i_j)=\tilde{x}^i_j$ for all $i,j$.
The relative moduli space
\begin{equation}\label{RelativeModuli}
\M^{V}_{\chi,m^1,\cdots,m^k}(D,d)
\end{equation}
consists of equivalence classes of   $V$-regular maps $(f,C;\{x^i_j\})$ with
the Euler characteristic $\chi(C)=\chi$ and with contact vectors $m^1,\cdots,m^k$.
Since no confusion can arise, we will regard a point in the space (\ref{RelativeModuli})
 as a $V$-regular map  $(f,C;\{x^i_j\})$. For simplicity,
we will often  write a $V$-regular map  $(f,C;\{x^i_j\})$ simply as $f$.

The (formal) complex dimension of the  space (\ref{RelativeModuli})   is
given by the left-hand side of the Riemann-Hurwitz formula (\ref{RHformula}):
\begin{equation}\label{DimRelModuli}
2d(1-h) -\chi  - \sum_{i=1}^k \big(\,d-\ell(m^i)\,\big).
\end{equation}
Suppose this dimension is zero. Then,  for each $V$-regula map  $(f,C;\{x^i_j\})$ in (\ref{RelativeModuli}), forgetting the contact marked points $x^i_j$
 gives   a (ramified) cover $f$ that is ramified only over fixed points $q^i$ and satisfies (\ref{RHformula}).
The automorphism group $\Aut(f)$ of a (ramified) cover $f$ consists of automorphisms $\sigma\in \Aut(C)$ with $f\circ \sigma=f$. The automorphism group $\Aut(f,V)$ of  a $V$-regular map $(f,C;\{x^i_j\})$
 consists of  automorphisms $\sigma\in \Aut(f)$ with
$\sigma(x^i_j)=x^i_j$ for all $i,j$.

For a partition $m$ of $d$, let  $\Aut(m)$ be the subgroup of symmetric group $S_{\ell(m)}$ permuting equal parts of the partition $m$.

\begin{lemma}\label{SHbyRelModuli}
Let $m^1,\cdots,m^k$ be  as above and
suppose the dimension  (\ref{DimRelModuli}) is zero.
\begin{itemize}
\item[(a)]
If $m^i=(1^d)$ for some $1\leq i\leq k$, then $\Aut(f,V)$ is trivial for all $f$ in (\ref{RelativeModuli}).
\item[(b)]
If $m^1,\cdots,m^k$ are all odd partitions, then
\begin{equation*}\label{SH-Rmoduli}
H^{h,p}_{m^1,\cdots,m^k}\ =\ \frac{1}{\prod_{i=1}^k |\Aut(m^i)|}\,\sum \,\frac{(-1)^{p(f)}}{|\Aut(f,V)|}
\end{equation*}
where the sum is over all  $f$ in (\ref{RelativeModuli}) and
$p(f)$ is the  parity (\ref{parity}).
\end{itemize}
\end{lemma}

\begin{proof}
Let $(f,C;\{x^i_j\})$ be a $V$-regular map in (\ref{RelativeModuli}) and $\sigma\in \Aut(f,V)$.
If $m^i=(1^d)$, then the set of branch points $B$ of $f$ is a subset of $V\setminus \{q^i\}$ and
the restriction of $\sigma$ to $C\setminus f^{-1}(B)$ is a covering transformation that fixes contact marked points
$x^i_1,\cdots,x^i_d$. Noting $f^{-1}(B)$ is finite, we conclude that $\sigma$ is an identity map on $C$. This proves (a).

As mentioned above, forgetting contact marked points $x^i_j$ gives a (ramified) cover $f$  satisfying (\ref{RHformula}).
Conversely, given a (ramified) cover $f$ satisfying (\ref{RHformula}),
one can mark a point over $q^i$ with ramification index $m^i_j$ as a contact marked point $x^i_j$.
Such marking gives $V$-regular maps $(f,C;\{x^i_j\})$ in $\prod_{i=1}^k |\Aut(m^i)|$ ways.
Observe that
$(f,C;\{x^i_j\})$ and $(f,C;\{\sigma(x^i_j)\})$ are isomorphic for each $\sigma\in\Aut(f)$ and that
$\Aut(f,V)$ is a normal subgroup of $\Aut(f)$.
Consequently,  the quotient group $G=\Aut(f)/\Aut(f,V)$ acts freely on the set of   $V$-regular maps $(f,C;\{x^i_j\})$ obtained by the (ramified) cover $f$.
Its orbits give $\prod_{i=1}^k |\Aut(m^i)|/|G|$ points (i.e. equivalence classes of $V$-regular maps) in the  space (\ref{RelativeModuli}), each of which has the same automorphism group $\Aut(f,V)$.
Now, (b) follows from counting maps with the  parity of map modulo automorphisms.
\end{proof}

\medskip

%%%%%%%%%%%%%%%%%%%%%%%%%%%%%%%%%%%%%%%%%%%%%%%%%%%%%%%%%%%%%%%%
%%%%%%%%%%%%%%%%%%%%%%%%%%%%%%%%%%%%%%%%%%%%%%%%%%%%%%%%%%%%%%%%
\setcounter{equation}{0}
\section{Maps into a nodal curve }
\label{section2}
\bigskip

Let $D_0=D_1\cup E \cup D_2$ be a connected nodal curve of (arithmetic) genus $h$
with two nodes $p^1$ and $p^2$ such that for $i=1,2$,
$E=\P^1$ meets $D_i$ at node $p^i$ and  $D_i$  has genus $h_i$ with $h_1+h_2=h$.
In this section, we consider maps into  $D_0$ that are relevant to our subsequent discussion.

\medskip
In the below, we fix $d$, $h$, $\chi$ and odd partitions $m^1,\cdots,m^k$ of $d$ so that the Riemann-Hurwitz formula (\ref{RHformula}) holds, or equivalently, the dimension formula (\ref{DimRelModuli}) is  zero.
For each partition $m$  of $d$,  consider the product space
\begin{equation*}\label{central-moduli}
\CP_m = \M^{V_1}_{\chi_1, (1^d),m^1,\cdots,m^{k_1},m}(D_1,d) \times  \M^{V_0}_{\chi_0, m,(1^d),m}(E,d) \times  \M^{V_2}_{\chi_2, m,m^{k_1+1},\cdots,m^k,(1^d)}(D_2,d)
\end{equation*}
where $V_1=\{q^{k+1},q^1,\cdots,q^{k_1},p^1\}$, $V_0=\{p^1,q^{k+2},p^2\}$, $V_2=\{p^2,q^{k_1+1},\cdots,q^k,q^{k+3}\}$ and
\begin{equation}\label{SumEulerC}
\chi_1+\chi_0+\chi_2-4\ell(m)\ =\ \chi.
\end{equation}
For simplicity, let $\M^1_m,\M_m^0$ and $\M_m^2$ denote the first, the second and the third factors of $\CP_m$.

\begin{lemma}\label{L:CModuli}
If $\CP_m\ne \emptyset$, then the spaces $\M^1_m,\M^0_m$ and $\M_m^2$ have dimension zero. Consequently,  $\chi_0=2\ell(m)$ and
$d-\ell(m)$ is even.
%$$
%\chi_0=2\ell(m)\ \ \ \ \ \ \ \mbox{and}\ \ \ \ \ \ \
%\ell(m)-d\ \ \mbox{is\ \ even}.
%$$
\end{lemma}

\begin{proof}
Each $\M^i_m$ ($0\leq i\leq 2$) has nonnegative dimension by the Riemann-Hurwitz formula.
The formula (\ref{SumEulerC}) and
our assumption that the dimension (\ref{DimRelModuli}) is zero thus imply that
each $\M^i_m$ has dimension zero. The dimension formulas  for $\M^0_m$ and $\M^i_m$ ($i=1,2$) then show that
$\chi_0=2\ell(m)$ and $d-\ell(m)$ is even because $d-\ell(m^i)=\sum (m^i_j-1)$ is even for all $1\leq i\leq k$.
\end{proof}

\medskip

Let $|A|$ denote the cardinality of a set $A$.

\begin{lemma}\label{NumM0}\ \
${\displaystyle |\M_m^0| \ =\   \frac{d!\,|\Aut(m)|}{\prod m_j}}$.
\end{lemma}

\begin{proof}
Let $f\in\M^0_m$. Since $\chi_0=2\ell(m)$, we have
\begin{itemize}
\item
the domain of $f$   is a disjoint union of smooth rational curves $E_j$ for $1\leq j\leq \ell(m)$,
\item
each restriction $f_j=f|_{E_j}$ has exactly one contact marked point over $p^i$ ($i=1,2$) with multiplicity $m_j$,
so $f_{j}$ has degree $m_j$.
\end{itemize}
Consequently, forgetting contact marked points of maps in $\M_m^0$ gives exactly one map (as a cover) with
automorphism group of order $|\Aut(m)|\prod m_j$. Here the factor $|\Aut(m)|$ appears because
we can relabel  maps $f_j$ in $|\Aut(m)|$ ways and the factor $\prod m_j$ appears because each restriction map $f_j$ (as a cover)
has an automorphism group of order $m_j$.
The argument in the proof of Lemma~\ref{SHbyRelModuli} then shows the lemma.
\end{proof}

For each $(f_1,f_0,f_2) \in \CP_m$, by identifying contact marked points over $p^i\in D^i\cap E$ ($i=1,2$), 
one can glue the domains of $f_i$ and $f_0$ to obtain a map $f:C\to D_0$ with $\chi(C)=\chi$.
For notational convenience, we will often write the glued map $f$ as $f=(f_1,f_0,f_2)$.
Denote by
\begin{equation}\label{CenterModuli}
\M_{m,0}
\end{equation}
the space of such glued maps $f=(f_1,f_0,f_2)$.
Contact marked points are labeled, but nodal points of $C$ are not labeled.
Thus, we have:

\begin{lemma}\label{GluingMap}
$\CP_m$ is a degree $|\Aut(m)|^2$ cover of $\M_{m,0}$.
\end{lemma}

\vskip 1cm

%%%%%%%%%%%%%%%%%%%%%%%%%%%%%%%%%%%%%%%%%%%%%%%%%%%%%%%%%%%%%%%%
%%%%%%%%%%%%%%%%%%%%%%%%%%%%%%%%%%%%%%%%%%%%%%%%%%%%%%%%%%%%%%%%
\setcounter{equation}{0}
\section{Limiting and gluing }
\label{section3}
\bigskip

Following \cite{IP2}, this section  describes limiting and gluing arguments under a degeneration of target curves.
Let $D_0=D_1\cup E\cup D_2$ be the nodal curve with fixed points $q^1,\cdots,q^{k+3}$ as in Section~\ref{section2}.
In Section~\ref{section4}, we will construct a family of curves together with $k+3$ sections:
\begin{equation}\label{deg-target}
\xymatrix{
 \D \ar[d]^{\rho} \\
 \Delta \ar@/^1pc/[u]^{Q^i}    }
\end{equation}
Here the total space $\D$ is a smooth complex surface, $\Delta\subset \cx$ is a disk with parameter $\la$, the central fiber is $D_0$, the general fiber $D_\la$ ($\la\ne 0$) is a smooth curve of genus $h$ and $Q^i(0)=q^i$ for  $1\leq i\leq k+3$.
By Gromov Convergence Theorem, a sequence of holomorphic maps into $D_\la$ with $\la\to 0$
has a map into $D_0$ as a limit.
For notational simplicity, for $\la\ne 0$ we set
\begin{equation}\label{def-Mr}
\M_\la\ =\ \M^{V_\la}_{\chi,m^1,\cdots,m^{k+3}}(D_\la,d)\ \ \ \ \
\mbox{where}\ \ \ \ \ V_\la\ =\ \{Q^1(\la),\cdots,Q^{k+3}(\la)\},
\end{equation}
and denote the set of limits of sequences of maps in $\M_\la$ as $\la\to 0$ by
\begin{equation}\label{LimitSet}
\lim_{\la\to 0} \,\M_\la. %\ \subset \ \bigsqcup_{\deg(s)=d} \M_s.
\end{equation}
Lemma~\ref{LimitModuli} below  shows that limit maps in (\ref{LimitSet}) lie
in the union of spaces (\ref{CenterModuli}), namely
\begin{equation}\label{E:LimitModuli}
\lim_{\la\to 0} \,\M_\la \ \subset \ \underset{m}{\tcup}\ \M_{m,0}
\end{equation}
where the union is over all partitions $m$ of $d$ with $d-\ell(m)$  even.

Conversely, by the Gluing Theorem of \cite{IP2}, the domain of each map in $\M_{m,0}$  can be smoothed to produce
maps in $\M_\la$ for small $|\la|$.
Shrinking $\Delta$ if necessary, for $\la\in \Delta$,   one can assign
to each $f_\la\in \M_\la$ a partition $m$ of $d$ by (\ref{E:LimitModuli}).
Let $\M_{m,\la}$  be the set of all  pairs $(f_\la,m)$. For each $f\in\M_{m,0}$,  let
\begin{equation}\label{ComFModuli}
\Z_{m,f}\ \to\ \Delta
\end{equation}
be the connected component of  $\bigcup_{\la\in\Delta} \M_{m,\la}\to \Delta$ that contains $f$ and let
\begin{equation}\label{main-moduli}
\Z_{m,f, \la}
\end{equation}
denote the fiber of (\ref{ComFModuli}) over $\la\in \Delta$.
It follows that for $\la\ne 0$
\begin{equation}\label{DecMr}
\M_\la\ =\ \underset{f\in \M_{m,0}}{\tsqcup} \Z_{m,f,\la}.
\end{equation}
For $f=(f_1,f_0,f_2)\in\M_{m,0}$ where $m=(m_1,\cdots,m_\ell)$,
let $y^i_j$ be the node mapped to $p^i$ at which
$f_i$ and $f_0$ have multiplicity $m_j$.
The Gluing Theorem shows that
one can smooth each node $y^i_j$,
in $m_j$ ways, to produce $(\prod m_j)^2$ maps in $\CZ_{m,f, \la}$, so we have
\begin{equation}\label{number-fiber}
|\CZ_{m,f, \la}|\ =\ \big(\,\tprod m_j\,\big)^2\ \ \ \ \ (\la\ne 0).
\end{equation}

\bigskip

In order to prove (\ref{E:LimitModuli}), we will use the following fact on stable maps.
An irreducible component of a stable holomorphic map $f$ is a ghost component if its image is a point.
Write the domain of $f$ as $C^g\cup C$ where $C^g$ is a connected curve whose irreducible components are all ghost components. Then the stability of $f$ implies that
\begin{equation}\label{stability}
\chi(C^g) -\ell^g - n \ \leq\ -1
\end{equation}
where $\ell^g=|C^g\cap C|$ and $n$ is the number of marked points on $C^g$.

\begin{lemma}\label{LimitModuli}
Let $\M_r$ and $\M_{m,0}$ be as above. Then we have
$$
\lim_{\la\to 0} \,\M_\la \ \subset \ \underset{m}{\tcup}\ \M_{m,0}
$$
where the union is over all partitions $m$ of $d$ with $d-\ell(m)$  even.
\end{lemma}

\begin{proof}
Let $f$ be a limit map in (\ref{LimitSet}). The domain $C$ of $f$ can be written as
\begin{equation}\label{decompo-domain}
C\ =\ C_1\ \tcup\ C_0\ \tcup\ C_2\ \ \tcup\ \big(\ \underset{i=1}{\overset{k+3}{\tcup}}\  C_i^g\ \big)\ \tcup \
C^g\ \tcup\ \tilde{C}^g
\end{equation}
where $C_0$ maps to $E$, $C_1$ and $C_2$ map to $D_1$ and $D_2$, $C_i^g$ is the union of all ghost components
over $q^i$ where $i=1,\cdots,k+3$, $C^g$ is the union of all ghost components
over points in  $D_0\setminus (V_1\cup V_0\cup V_2)$ and $\tilde{C}^g$ is the union of all ghost components  over $\{p^1,p^2\}$.
Let $f_j=f|_{C_j}$ for $j=0,1,2$.
Observe that $f_j$ is $V_j$-regular because
$C_j$ has no ghost components.
Let $\hat{m}^i$ be a contact vector over $q^i$,
$\tilde{m}^1$ and $\tilde{m}^2$ be contact vectors of $f_1$ and $f_2$ over  $p^1$ and $p^2$,
and $\tilde{m}^{0;1}$ and  $\tilde{m}^{0;2}$ be contact vectors of $f_0$ over $p^1$ and $p^2$.
The Riemann-Hurwitz formulas for $f_0$, $f_1$ and  $f_2$ give
\begin{equation}\label{RH-C-hat}
\sum_{j=0}^2 \chi(C_j)\ \leq\ 2d(1-h)  + \sum_{i=1}^{k+3} \big(\,\ell(\hat{m}^i)-d\,\big)
+ \sum_{i=1}^2 \big(\,\ell(\tilde{m}^i)+  \ell(\tilde{m}^{0;i})\,\big).
\end{equation}

For $i=1,\cdots, k+3$,
let $\ell_i=|C_1\cup C_0 \cup C_2 \cap C_i^g|$ and let $n_i$ be the number of marked points on $C_i^g$.
Since  all  marked points are limits of  marked points, we have
\begin{equation}\label{number-marked-pts}
\ell(\hat{m}^i)\  = \ \ell(m^i) - n_i + \ell_i.\ \ \ \ \
%\mbox{where}\ \ \ \ \ i=1,\cdots,k+3.
\end{equation}
For $j=0,1,2$, let $\tilde{\ell}_j=|C_j\cap \tilde{C}^g|$.
Counting the number of nodes mapped to $p^1$ and $p^2$ shows
\begin{equation}\label{CountInt}
\sum_{i=1}^2\,\big(\,\ell(\tilde{m}^i)-\tilde{\ell}_i\,\big)\ =\ \sum_{i=1}^2\,|C_i\cap C_0| \ =\
\sum_{i=1}^2\,\ell(\tilde{m}^{0;i}) - \tilde{\ell}_0.
\end{equation}
Let $\ell^g=|C_1\cup C_0 \cup C_2 \cap C^g|$.
Since $\chi(C)=\chi$, by (\ref{decompo-domain}) and (\ref{CountInt}) we have
\begin{equation}\label{EuleC-domain}
\chi\ =\ \sum_{j=0}^2 \chi(C_j) + \sum_{i=1}^{k+3} \big(\,\chi(C_{i}^g) -2\ell_i\,\big)
+\chi(C^g) -2\ell^g + \chi(\tilde{C}^g)
- \tilde{\ell} - \sum_{i=1}^2 \big(\,\ell(\tilde{m}^i)+\ell(\tilde{m}^{0;i})\,\big)
\end{equation}
where $\tilde{\ell}=\tilde{\ell}_0+\tilde{\ell}_1+\tilde{\ell}_2$.
By our assumption that the formula (\ref{RHformula}) holds,
it follows from (\ref{RH-C-hat}), (\ref{number-marked-pts}) and (\ref{EuleC-domain}) that
\begin{equation}\label{EulerCha}
\chi\
\leq\
\chi + \sum_{i=1}^{k+3} \big(\,\chi(C_i^g)-   \ell_i-n_i\,\big)  + \chi(C^g)  -2\ell^g  +
\chi(\tilde{C}^g) -  \tilde{\ell}.
\end{equation}

Noting $C^g$ and $\tilde{C}^g$ have no marked points, by (\ref{stability}) and (\ref{EulerCha}) we conclude
that the domain $C$ of $f$ has no ghost components. Consequently,
\begin{itemize}
\item
$f_j$ is $V_j$-regular for $j=0,1,2$,
\item
$\tilde{m}^i=\tilde{m}^{0;i}$ for $i=1,2$ (cf. Lemma 3.3 of \cite{IP2}) and  $\hat{m}^i=m^i$ for $i=1,\cdots,k+3$.
\end{itemize}
In particular, the equality in (\ref{RH-C-hat}) holds; otherwise we have a strict inequality in (\ref{EulerCha}).
So, we have $\chi(C_0)=\ell(\tilde{m}^1)+\ell(\tilde{m}^2)$.
But  $\chi(C_0) \leq 2\min\{\ell(\tilde{m}^1),\ell(\tilde{m}^2)\}$.
It follows that
\begin{itemize}
\item
$C_0$ has $\ell(\tilde{m}^1)=\ell(\tilde{m}^2)$ connected components $E_j$ with $\chi(E_j)=2$ for all $j$,
\item
$\tilde{m}^1_j = \deg(f_0|_{E_j}) = \tilde{m}^2_j$ for all $j$, i.e., $\tilde{m}^1=\tilde{m}^2$.
\end{itemize}
It follows that the Euler characteristics of
$C_0$, $C_1$ and $C_2$ satisfy (\ref{SumEulerC}) by (\ref{EuleC-domain}).
Therefore,
$f\in\M_{m,0}$ for $m=\tilde{m}^1=\tilde{m}^2$ and $d-\ell(m)$ is even by Lemma~\ref{L:CModuli}.
\end{proof}

\vskip 1cm

%%%%%%%%%%%%%%%%%%%%%%%%%%%%%%%%%%%%%%%%%%%%%%%%%%%%%%%%%%%%
%%%%%%%%%%%%%%%%%  Section 4 ---  %%%%%%%%%%%%%%%%%%%%%%%%
%%%%%%%%%%%%%%%%%%%%%%%%%%%%%%%%%%%%%%%%%%%%%%%%%%%%%%%%%%%%%%

\setcounter{equation}{0}
\section{Smooth model  by Schiffer variation}
\label{section4}
\bigskip
A {\em Schiffer Variation} of a nodal curve (cf. pg. 184 of \cite{ACG}) is
obtained by gluing deformations $uv=\la$ near nodes with the trivial deformation away from  nodes.
In this section, we use the method of Schiffer variation  to construct a smooth model for
the space $\Z_{m,f}$ in (\ref{ComFModuli}) which has several branches intersecting at $f$ unless $m$ is trivial.

\medskip
Throughout this section,  we fix an odd partition
$m=(n^\ell)$, i.e. $m=(m_1,\cdots,m_\ell)$  with
\begin{equation}\label{assumption-n}
m_1\ =\ \cdots\ =\ m_\ell\ =\ n\ \ \ \ \ \mbox{where}\ \  n=d/\ell\ \ \mbox{is\ odd}.
\end{equation}
Let $f=(f_1,f_0,f_2)$ be a map in $\M_{m,0}$ in (\ref{CenterModuli}).
As described in Section~2,
the central fiber of $\rho:\D\to \Delta$ is the nodal curve $D_0= D_1 \cup E \cup D_2$ with two nodes $p^1\in D_1\cap E$ and
$p^2\in D_2\cap E$ where $E=\P^1$.
The domain of $f$ is a nodal curve
$$
C  \ = \  C_1 \cup C_0  \cup  C_2
\ \ \ \ \ \mbox{where}\ \ \ \ \ C_0=\cup_{j=1}^\ell E_\ell
$$
with $2\ell$ nodes such that for $i=1,2$ and $j=1,\cdots,\ell$,
\begin{itemize}
\item
$f^{-1}(p^i)$   consists of the $\ell$ nodes $y^i_j\in C_i\cap E_j$,
\item
$C_i$ is smooth and $f|_{C_i}=f_i$ has ramification index $m_j=n$ at the node $y^i_j$,
\item
$E_j=\P^1$ and $f|_{E_j}=f_0|_{E_j}:E_j\to E$ has ramification index $m_j=n$ at the node $y^i_j$.
\end{itemize}

\medskip

The following is a main result of this section.

\begin{prop}\label{TwistingProp}
Let $f$ be as above.
Then, for each vector $\zeta=(\zeta_1^1,\zeta_1^2,\cdots,\zeta_\ell^1,\zeta_\ell^2)$ where $\zeta_j^i$
is a $n$-th root of unity,
there are a family of curves $\varphi_\zeta:\C_\zeta\to \Delta$, with smooth total space $\C_\zeta$,
over a disk $\Delta$ (with parameter $s$) and a holomorphic map $\F_\zeta:\C_\zeta\to \D$
satisfying:
\begin{itemize}
\item[(a)]
The central fiber $C_{\zeta,0}=C$ and the restriction map  $\F_\zeta|_{C} = f$.
\item[(b)]
The general fiber $C_{\zeta,s}$ ($s\ne 0$) is smooth and for $\la=s^n\ne 0$
\begin{equation}\label{SM-Fmoduli}
\underset{\zeta}{\tcup}\ \{\,f_{\zeta,s}\}\ =\ \CZ_{m,f,\la}
\end{equation}
where the union is over all $\zeta$, $f_{\zeta,s}=\F_\zeta|_{C_{\zeta,s}}$ and $\Z_{m,f,\la}$ is the space (\ref{main-moduli}).
\end{itemize}
\end{prop}

\begin{proof} The proof consists of 4 steps.

\medskip
\non
{\bf Step 1 :}
We first  show how to  construct the family of curves
$\rho:\D\to\Delta$    with $k+3$ sections.
For $i=1,2$, a neighborhood  of the node  $p^i\in D_i\cap E$
 can be regarded  as the union $U^i\cup V^i$ of the  two disks
$$
U^i \ =\  \{\,u^i\in\cx \,:\, |u^i|<1\,\}\ \subset\ D_i\ \ \ \ \
\mbox{and}\ \ \ \ \
V^i\ =\  \{\,v^i\in\cx \,:\, |v^i|<1\,\}\ \subset\ E
$$
with their origins identified. We may assume that the fixed points $q^1, \dots, q^{k+3}$ in $D_0$
described above (\ref{SumEulerC}) lie outside these sets. Consider the regions
\begin{align*}
  A^i \  &=\  \{\, (u^i,v^i,\la)\,\in\, U^i\times V^i \times \Delta \ : \  u^i v^i=\la\, \},\\
B   \  &= \ \overset{2}{\underset{i=1}{\tcup}}\, G^i\ \ \tcup\ \
\big[\, \big(\, D_0 \setminus \ \overset{2}{\underset{i=1}{\tcup}}\,  (\,U^i \ \tcup\  V^i\,) \, \big) \times \Delta\,
\big]
\end{align*}
where
$$
G^i\ =\
\big\{\,(u^i,\la)\in U^i\times \Delta\,:\,|u^i|>\sqrt{|\la|}\,\} \ \ \tcup\ \
\{\,(v^i,\la)\in V^i\times \Delta\,:\,|v^i|>\sqrt{|\la|}\ \big\}.
$$
We  obtain a smooth complex surface $\D$  by gluing $A^1,A^2$ and $B_0$
using the maps
\begin{equation}\label{Gluing1}
G^i \  \to\  A^i
\ \ \  \mbox{defined\ by}\ \ \
(u^i,\la) \to \big( \,u^i,  \tfrac{\la}{u^i}, \la\,\big)
\ \ \mbox{and}\ \
(v^i,\la) \to \big(\,\tfrac{\la}{v^i}, v^i, \la\,\big).
\end{equation}
Let  $\rho:\D\to \Delta$   be the projection to the last factor  and  define $k+3$ sections $Q^i$ of $\rho$ by
$$
Q^i(\la) \ =\  (q^i,\la).
$$

\non
{\bf Step 2 :} We  can similarly construct a family of curves over a $2\ell$-dimensional polydisk:
\begin{equation}\label{def-domain}
\varphi_{2\ell}: \X \ \to\  \Delta_{2\ell}\,=\,\{\,t=(t^1_1,t^2_1,\cdots,t^1_\ell,t^2_\ell)\in\cx^{2\ell} \,:\,|t^i_j|<1\,\}.
\end{equation}
For each node  $y^i_j\in C_i\cap E_j$, choose a neighborhood obtained from two disks
$$
U^i_j \ =\  \{\,u^i_j\in\cx \,:\, |u^i_j|<1\,\}\ \subset\ C_i\ \ \ \ \
\mbox{and}\ \ \ \ \
V^i_j \ =\  \{\,v^i_j\in\cx \,:\, |v^i_j|<1\,\}\ \subset\ E_j
$$
by identifying the origins.
Consider the regions
\begin{align*}
A^i_j &\ =\  \{\, (u^i_j,u^i_j,t)\,\in\, U^i_j\times V^i_j \times \Delta_{2\ell} \ : \ u^i_jv^i_j=t^i_j\, \}, \\
B_{2\ell}   &\ =\  \underset{i,j}{\tcup}\,G^i_j\  \ \tcup\  \
\big[\,\big(\,C \setminus \  \underset{i,j}{\tcup} \ (\,U^i_j \ \tcup \ V^i_j\,) \,\big) \times \Delta_{2\ell}\,\big]
^{\phantom{\displaystyle A}}
\end{align*}
where
$$
G^i_j\ =\
\big\{\,(u^i_j,t)\in U^i_j\times \Delta_{2\ell}\ :\ |u^i_j|>\sqrt{|t^i_j|}\, \ \big\} \ \ \tcup\ \
\big\{\,(v^i_j,t)\in V^i_j\times \Delta_{2\ell}\ :\ |v^i_j|>\sqrt{|t^i_j|}\,\ \big\}.
$$
We can then obtain a smooth complex manifold $\X$  of dimension $2\ell+1$ by gluing $\cup A^i_j$ and $B_{2\ell}$
with the maps
\begin{equation}\label{Gluing2}
G^i_j \ \to\  A^i_j
\ \ \ \mbox{defined\ by}\ \ \
(u^i_j,t)\ \to\ \big(\, u^i_j, \tfrac{t^i_j}{u^i_j}, t \,\big)
\ \  \mbox{and}\ \
(v^i_j,t)\ \to\ \big(\, \tfrac{t^i_j}{v^i_j}, v^i_j, t \,\big).
\end{equation}
Let $\varphi_{2\ell}:\X\to \Delta$   be the projection to the factor $t$.

\medskip

\non
{\bf Step 3 :}
Since $f_i$ and $f_0|_{E_j}$ have ramification index $m_j=n$ at $y^i_j$,
we may assume (after coordinates change) that on $U^i_j$ and $V^i_j$ the map $f$ can be written as
\begin{equation}\label{assumption-map}
U^i_j \ \to\ U^i\ \ \mbox{by}\ \ u^i_j \ \to\ (u^i_j)^n
%u^i\ =\ (u^i_j)^{s_j}\ \ \ \ \
\ \ \ \ \ \mbox{and}\ \ \ \ \
V^i_j \ \to\ V^i\ \ \mbox{by}\ \ v^i_j\ \to\ (v^i_j)^n.
%v^i\ =\ (v^i_j)^{s_j}
\end{equation}
For each $i,j$, define a map
\begin{equation}\label{LocalMap1}
G^i_j\ \to\ G^i\ \ \ \ \ \mbox{by}\ \ \ \ \
(u^i_j,t)\ \to \big(\,(u^i_j)^n,(t^i_j)^n\,\big)
\ \  \mbox{and}\ \
(u^i_j,t)\ \to \big(\,(v^i_j)^n,(t^i_j)^n\,\big).
\end{equation}
On the other hand, for each $i,j$, we have a map
\begin{equation}\label{LocalMap2}
A^i_j\ \to\ A^i\ \ \ \mbox{defined\ by}\ \ \
(u^i_j,v^i_j,t)\ \to\ \big(\,(u^i_j)^n,(v^i_j)^n,(t^i_j)^n\,\big).
\end{equation}
These two maps (\ref{LocalMap1}) and (\ref{LocalMap2}) are glued together under the maps (\ref{Gluing1}) and
(\ref{Gluing2}). The glued map extends to a holomorphic map
$f_t:\X_t \to D_\la$ if and only if
\begin{equation}\label{matching}
(t^1_1)^n \ =\  (t^2_1)^n \ = \   \cdots\ = \ (t^1_\ell)^n \ = \ (t^2_\ell)^n \ =\  \la.
\end{equation}
There are $n^{2\ell}$ solutions $t$ of (\ref{matching}) and the extension map $f_t$ is given by
\begin{equation*}\label{LocalMap3}
(x,t)\ \to\  (f(x),\la)\ \ \ \ \ \mbox{on}\ \ \ \ \ \X_t \ -\  \tcup \,A^i_j.
\end{equation*}

\medskip
\non
{\bf Step 4 :}
For each vector $\zeta=(\zeta_1^1,\zeta_1^2,\cdots,\zeta_\ell^1,\zeta_\ell^2)$ where each $\zeta_j^i$
is a $n$-th root of unity, define
$$
\delta_\zeta : \Delta \to \Delta_{2\ell}\ \ \ \ \
\mbox{by}\ \ \ \ \
s\ \to\ (\,\zeta_1^1s,\zeta_1^2s,\zeta_2^1s,\zeta_2^2s,\,\cdots\,,\zeta_\ell^1s,\zeta_\ell^2s\,).
$$
The pull-back $\delta_\zeta^*\X$ gives a family of curves:
\begin{equation}\label{smooth-family}
\xymatrix{
 \C_\zeta= \delta_\zeta^*\X  \ar[rr]  \ar[d]_{\varphi_\zeta}
 &&  \X  \ar[d]^{\varphi_{2\ell}} \\
 \Delta  \ar[rr]^{\delta_\zeta}
 && \Delta_{2\ell}      }
\end{equation}
The central fiber is $C_{\zeta,0}=C$ and the general fiber $C_{\zeta,s}$ ($s\ne 0$) is smooth.
A neighborhood of the node $y^i_j$ of $C$ in $\C_\zeta$ can be viewed as
\begin{equation}\label{nbd-beforeBU}
\wh{A}^i_j \ =\  \{\,(u^i_j,v^i_j,s)\in \cx^3 \,: \, |u^i_j|<1,\ |v^i_j|<1,\ u^i_jv^i_j=\zeta^i_j s\,\}.
\end{equation}
It follows that the total space $\C_\zeta$ is a complex smooth surface.
Noting $\delta_\zeta(s)$ is a solution of (\ref{matching}) for $\la=s^n$, we obtain a holomorphic map
$\F_\zeta:\C_\zeta\to \D$ given by
\begin{equation}\label{Def-F}
\begin{array}{rllll}
(u^i_j,v^i_j,s)\ &\to\ \ &(\,(u^i_j)^n,(v^i_j)^n,s^n\,)\ \ \  &\mbox{on}\   &\wh{A}^i_j, \\
(x,s)\ &\to\ &(\,f(x),s^n\,)\ \ \ \ \  \ \ \ &\mbox{on}\   &\C_\zeta\ -\ \tcup\, \wh{A}^i_j.
\end{array}
\end{equation}

Since the restriction $\F_\zeta|_C=f$ by (\ref{assumption-map}) and (\ref{Def-F}),
it remains to show (\ref{SM-Fmoduli}). By our choice of fixed points $q^i$ on $D_0$,
each contact marked point $x^i_j$ of $f$ lies in $\C_\zeta - \cup \wh{A}^i_j$.
Thus, by (\ref{Def-F}), the pull-back $\F_\zeta^*Q^i$ of the section $Q^i$ of $\rho$ gives
a section $X^i_j$ of $\varphi_\zeta$ given by $X^i_j(s)=(x^i_j,s)$.
After marking the points $X^i_j(s)$ in $C_{\zeta,s}$,
the restriction map
\begin{equation*}\label{RestrictionMap}
f_{\zeta,s}\ =\ \F_\zeta|_{C_{\zeta,s}}:C_{\zeta,s}\ \to \ D_\la\ \ \ \ \ \mbox{where}\ \ \ \ \ \la=s^n\ne 0
\end{equation*}
has contact marked points
$X^i_j(s)$ over $Q^i(\la)$ with multiplicity $m^i_j$. This means $f_{\zeta,s}$ lies in the space $\M_\la$ in
(\ref{def-Mr}) for $\la=s^n$.
Therefore, noting  (i) $f_{\zeta,s}\to f$ as $s\to 0$ and
(ii) $|\Z_{m,f,\la}|=n^{2\ell}$  by (\ref{number-fiber}),
we conclude (\ref{SM-Fmoduli}). This completes the proof.
\end{proof}

\medskip

%\begin{rem}\label{decent-curve}
%In Step 1, we may assume $v^1=1/v^2$ away from the nodes $p^1$ and $p^2$
%because $E=\P^1$. Then, since $u^1v^1=\la=u^2v^2$,  there is a neighborhood of $E$ obtained by blowing up the  surface
%$u^1u^2=\la^2$  at the origin.  The family $\rho:\D\to \Delta$ is called a family of decent curves and
%$E$ is an exceptional component  (cf. pg. 561 of \cite{C}).
%In (\ref{nbd-beforeBU}),
%we may also assume $v^1_j=1/v^2_j$ away from the nodes $y^1_j$ and $y^2_j$ for $j=1,\cdots,\ell$.
%Since we can write $\hat{u}^1_jv^i_j=s=\hat{u}^2_jv^i_j$ where $\hat{u}^i_j=u^i_j/\zeta^i_j$ for $i=1,2$,
%the family $\varphi_\zeta:\C_\zeta\to \Delta$ is also a family of decent curves and
%each $E_j$ is an exceptional component of $C$.
%\end{rem}

\vskip 1cm

%%%%%%%%%%%%%%%%%%%%%%%%%%%%%%%%%%%%%%%%%%%%%%%%%%%%%%%%%%%%%%%%
%%%%%%%%%%%%%%%%%%%%%%%%%%%%%%%%%%%%%%%%%%%%%%%%%%%%%%%%%%%%%%%%
\setcounter{equation}{0}
\section{Spin structure and parity}
\label{section5}
\bigskip

The aim of this section is to use a spin structure on a family of nodal curves \cite{C}
to show parity calculation in Proposition~\ref{P:Main-Parity} below.
Twisting bundle as in
(\ref{Bdl-induced}) below is a key idea for parity calculation.

\medskip

We first introduce a spin structure
on a family of nodal curves that is relevant to our discussion.
We refer to \cite{C} for the definition of spin structure and more details.
The relative dualizing sheaf $\omega_\rho$ of the family of curves $\rho:\D\to \Delta$  in (\ref{deg-target}) is
the canonical bundle $K_\D$ on the total space $\D$ since  $\D$  is smooth and $K_{\!\Delta}$ is trivial.
For each $\la\ne 0$, the restriction $K_\D|_{D_\la}$  is the canonical bundle $K_{D_\la}$  on $D_\la$ and
the restriction $K_\D|_{D_0}$ is the dualizing sheaf $\omega_{D_0}$ of the nodal curve $D_0=D_1\cup E\cup D_2$.
As described in Section~4,  $D_0$ is locally given by  $u^iv^i=0$ near each node $p^i$ in $D_i\cap E$ for $i=1,2$.
Then the local generators of $\omega_{D_0}$ are $du^i/u^i$ and $dv^i/v^i$ with a relation $du^i/u^i+dv^i/v^i=0$ (cf. page 82 of \cite{HM}).
This implies the restriction $\omega_{D_0}|_{D_i}=K_{D_i}\otimes \O(p^i)$.
On the other hand, $1/u^i$ is a local defining function for the divisor $-E$ on $\D$ near $p^i$. By restricting $1/u^i$ to $D_i$,
one can see that
$\O(-E)|_{D_i}=\O(-p^i)$.
Consequently, for $i=1,2$
\begin{equation}\label{res-can}
K_\D|_{D_i}\otimes \O(-E)|_{D_i}\ =\ \omega_{D_0}|_{D_i}\otimes \O(-p^i)\ =\ K_{D_i}.
\end{equation}

From Cornalba's construction (cf. pg. 570 of \cite{C}),  there are  a
line bundle $\N\to\D$ and  a homomorphism $\Phi:\N^2\to \omega_\rho=K_\D$ satisfying:
\begin{itemize}
\item
$\Phi$ vanishes identically on the exceptional component $E$ and $\N|_E = \O_E(1)$.
\item
Since $\Phi|_{E}\equiv 0$, there is an induced homomorphism
$\hat{\Phi}:\N^2\ \to \ K_\D\otimes \O(-E)$
such that $\Phi$ is the composition of
$\hat{\Phi}$ and tensoring with $\eta$:
\begin{equation}\label{d-i-h}
\Phi: \N^2\ \overset{\hat{\Phi}}{\longrightarrow}\  K_\D\otimes \O(-E)\
\overset{\otimes \eta}{\longrightarrow}\ K_\D
\end{equation}
where  $\eta$ is a section of $\O(E)$ with zero divisor $E$.
Then, for $i=1,2$, the restriction
$$
\hat{\Phi}|_{D_i} : (\N|_{D_i})^2\ \to\ K_\D|_{D_i}\otimes \O(-E)|_{D_i}=K_{D_i}
$$
is an isomorphism so that the restriction $N_i=\N|_{D_i}$ is a theta
characteristic on $D_i$.
\item
For each $\la\ne 0$, the restriction  $\Phi|_{D_\la}: (\N|_{D_\la})^2\to  K_{D_\la}$
is an isomorphism so that the restriction $N_\la= \N|_{D_\la}$ is a theta characteristic on $D_\la$.
\end{itemize}
The pair $(\N,\Phi)$ is a spin structure on $\rho:\D\to\Delta$ and
the restriction $\N|_{D_0}$ is a theta characteristic on the nodal curve $D_0$.

\begin{rem}\label{rem-parity}
Atiyah \cite{A} and Mumford \cite{M} showed that the parity of a theta characteristic on a smooth curve
is a deformation invariant.
Cornalba used the homomorphism $\Phi$ to extend Mumford's proof  to
the case of spin structure on a family of nodal curves (see pg. 580 of \cite{C}).
Thus, if $p_1,p_2$ and $p$ are the parities of $N_1,N_2$ and $N_\la$ ($\la\ne 0$),  then we have
$$
p \ \equiv\  p_1+p_2\ \  (mod\  2).
$$
\end{rem}

\medskip

Let $\varphi_\zeta:\C_\zeta \to \Delta$ be the family of curves  in Proposition~\ref{TwistingProp}.
Recall that the central fiber of $\varphi_\zeta$ is $C=C_1\cup C_0\cup C_2$ where $C_0=\sqcup_j E_j$ is a disjoint union of $\ell$ exceptional components $E_j$ and $C_i\cap E_j=\{y^i_j\}$ for $i=1,2$ and $1\leq j\leq \ell$.
Similarly as for (\ref{res-can}),
by restricting local defining functions, we have
\begin{equation}\label{tw-res1}
\O(\pm C_0)|_{C_i}\ =\ \O\big(\pm\underset{j}{\tsum} \,y^i_j\,\big)
\ \ (i=1,2)\ \ \ \ \
\mbox{and}\ \ \ \ \
\O(\pm C_0)|_{C_{\zeta,s}}\ =\ \O\ \  (s\ne 0).
\end{equation}
Since any fiber of $\varphi_\zeta$ is a principal divisor on $\C_\zeta$,
$\O(C)=\O$ and hence
$\O(C_0) = \O(-C_1-C_2)$. We also have
\begin{equation}\label{tw-res2}
\O(\pm C_0)|_{E_j}\ =\ \O(\mp (C_1+C_2))|_{E_j}\ =\ \O(\mp(y^1_j+y^2_j))\ =\ \O(\mp 2) \ \
(1\leq j\leq \ell).
\end{equation}

Let $f=(f_1,f_0,f_2)$ and $\F_\zeta:\C_\zeta\to\D$ be the maps in Proposition~\ref{TwistingProp}.
The  ramification divisor $R_{\F_\zeta}$ of $\F_\zeta$ has local defining functions given by the Jacobian of $\F_\zeta$,
so (\ref{Def-F}) shows
\begin{equation}\label{R-divisor}
R_{\F_\zeta}\ =\ \O(\,X_\zeta+(n-1)\,C )\ =\
\O(X_\zeta)
\end{equation}
where $X_\zeta= \sum_{i,j} (m^i_j-1) X^i_j$ and $X^i_j$ is the section of $\varphi_\zeta$ defined below (\ref{Def-F}).
%Recall that $\F_\zeta|_{C_1}=f_1$ and  $\F_\zeta|_{C_1}=f_1$. Let $f_{\zeta,s}=\F_\zeta|_{C_{\zeta,s}}$ for $s\ne 0$.
Note that
\begin{itemize}
\item[(i)]
the ramification divisor of $f_i=\F_\zeta|_{C_i}$ ($i=1,2$) is  $R_{f_i}=X_\zeta|_{C_i}+ \sum_j (n-1) y^i_j$,
\item[(ii)]
the ramification divisor of $f_{\zeta,s}=\F_\zeta|_{C_{\zeta,s}}$ ($s\ne 0$) is
$R_{f_{\zeta,s}}=X_\zeta|_{C_{\zeta,s}}$.
\end{itemize}

Now, noting $n$ is odd, we twist the pull-back bundle $\F_\zeta^*\N$ by setting
\begin{equation}\label{Bdl-induced}
\L_\zeta\ =\ \F_\zeta^*\N \otimes \O\big(\,\tfrac12\, X_\zeta +  \tfrac{(n-1)}{2}\, C_0\,\big).
\end{equation}

The lemma below shows that the twisted line $\L_\zeta$ restricts to a theta characteristic on each fiber of $\varphi_\zeta$,
including the central fiber $C$.
%Twisting by $\O(\frac{n-1}{2}C_0)$ makes  $\L_\zeta|_{C}$ a theta characteristic on the nodal curve  $C$.

\begin{lemma}\label{TC-induced}
Let $\L_\zeta$ be as above. Then, we have
\begin{itemize}
\item[(a)]
$\L_\zeta|_{E_j} = \O(1)$ for $1\leq j\leq \ell$,
\item[(b)]
$\L_\zeta|_{C_1} = L_{f_1}$,  $\L_\zeta|_{C_2} = L_{f_2}$ and $\L_\zeta|_{C_{\zeta,s}} = L_{f_{\zeta,s}}$ for $s\ne 0$
\end{itemize}
where $L_{f_1}$, $L_{f_2}$  and $L_{f_{\zeta,s}}$ are the theta characteristics
on $C_1$, $C_2$ and $C_{\zeta,s}$
defined  by (\ref{TwistedTheta}).
\end{lemma}

\begin{proof}
(a) follows from  (\ref{tw-res2}) and the fact that each restriction map  $\F_\zeta|_{E_j}$ has degree $n$.
(b) follows from (\ref{tw-res1}), (i) and  (ii).
\end{proof}

\medskip

Observe that the relative dualizing sheaf $\omega_{\varphi_\zeta}$  is the canonical bundle
$K_{\C_\zeta}$ since $\C_\zeta$ is smooth. The Hurwitz formula and (\ref{R-divisor}) thus imply that
\begin{equation}\label{Can-family}
\omega_{\varphi_\zeta}\ =\ K_{\C_\zeta}\ =\ \F^*_\zeta K_\D\otimes \O(X_\zeta).
\end{equation}
Define a homomorphism
\begin{equation}\label{HomE2}
\hat{\Psi}_\zeta : \L_\zeta^2=\F_\zeta^*\N^2\otimes  \O(X_\zeta+(n-1)C_0)\ \to\ \
\F_\zeta^*(K_\D\otimes \O(-E)) \otimes \O(X_\zeta+(n-1)C_0)
\end{equation}
by $\hat{\Psi}_\zeta=\F_\zeta^*\hat{\Phi}\otimes Id$ where $\hat{\Phi}$ is the induced homomorphism in (\ref{d-i-h}).
Noting $\O(C)=\O$ and $\O(D_0)=\O$, by  (\ref{Def-F}) we have
\begin{equation*}\label{pull-back-E}
\F^*_\zeta\O(-E)\ =\ \F_\zeta^*\O(D_1+D_2)\  =\
\O(n(C_1+C_2))\ =\ \O(-nC_0).
\end{equation*}
Together with (\ref{Can-family}), this implies that
the right-hand side of (\ref{HomE2}) is $K_{\C_\zeta} \otimes \O(-C_0)$.
Now, define a homomorphism $\Psi_\zeta:\L_\zeta^2\to K_{\C_\zeta}$ to be the composition
\begin{equation}\label{def-hom}
\Psi_\zeta: \L_\zeta^2\ \overset{\hat{\Psi}_\zeta}{\longrightarrow}\ K_{\C_\zeta} \otimes \O(-C_0)\
\overset{\otimes \xi}{\longrightarrow}\ K_{\C_\zeta}
\end{equation}
where $\xi$ is a section of $\O(C_0)$ with zero divisor $C_0$.

\begin{lemma}\label{SS-induce}
$(\L_\zeta,\Psi_\zeta)$ is a spin structure on $\varphi_\zeta:\C_\zeta\to \Delta$.
\end{lemma}

\begin{proof}
First, $\L_\zeta|_{E}=\O(1)$ by  Lemma~\ref{TC-induced}\,(a) and
$\Psi_\zeta$ vanishes identically on each exceptional component $E_j$ since $\xi=0$ on $C_0=\sqcup_j E_j$.
Second, since $\hat{\Phi}|_{D_i}$ is an isomorphism,  (\ref{tw-res1}) and (i) show that for $i=1,2$ the restriction
$$
\hat{\Psi}|_{C_i} = f_i^*(\hat{\Phi}|_{D_i})\otimes Id : (\L_\zeta|_{C_i})^2=f_i^*N_i^2\otimes \O(R_{f_i}) \to\
f_i^*K_{D_i}\otimes \O(R_{f_i})= K_{C_i}
$$
is an isomorphism. Lastly,
let $\la=s^n\ne 0$. Since $\Phi|_{D_\la}$ is an isomorphism, so is $\hat{\Phi}|_{D_\la}$.
Thus, by (\ref{tw-res1}), (ii) and the facts $K_\D|_{D_\la}=K_{D_\la}$ and $\O(-E)|_{D_\la}=\O$,
the restriction
$$
\hat{\Psi}_\zeta|_{C_{\zeta,s}}=f_{\zeta,s}^*\hat{\Phi}|_{D_\la}\otimes Id:
(\L_\zeta|_{C_{\zeta,s}})^2=f_{\zeta,s}^*N_\la^2\otimes\O(R_{f_{\zeta,s}})\ \to\
f_{\zeta,s}^*K_{D_\la}\otimes \O(R_{f_{\zeta,s}})=K_{C_{\zeta,s}}
$$
is an isomorphism. This implies that the restriction
$$
\Psi_\zeta|_{C_{\zeta,s}}:(\L_\zeta|_{C_{\zeta,s}})^2\ \to\  K_{C_\zeta}|_{C_{\zeta,s}}=K_{C_\zeta,s}
$$
is also an isomorphism.
Therefore, we conclude that $(\L_\zeta,\Psi_\zeta)$ is a spin structure on $\varphi_\zeta$.
\end{proof}

\medskip

The following is a key fact for the proof of Theorem~\ref{Main} in the Introduction.

\begin{prop}\label{P:Main-Parity}
Let $f=(f_1,f_0,f_2)$ and $f_{\zeta,s}$ be maps in Proposition~\ref{TwistingProp}. Then, for all $s\ne 0$
\begin{equation}\label{main-parity}
p(f_{\zeta,s}) \equiv\  p(f_1)+p(f_2)\ \ \ \ \ \mbox{(mod\  2)}.
\end{equation}
\end{prop}

\begin{proof}
Since $(\L_\zeta,\Psi_\zeta)$ is a spin structure on $\varphi_\zeta$,
the Cornalba's proof, mentioned in Remark~\ref{rem-parity}, shows that for all $s\ne 0$
$$
h^0(\L_\zeta|_{C_{\zeta,s}})\ \equiv\ h^0(\L_\zeta|_{C_1})+h^0(\L_\zeta|_{C_2})
\ \ \ \ \ \mbox{(mod\ 2)}.
$$
This and Lemma~\ref{TC-induced}\,b prove (\ref{main-parity}).
\end{proof}

\vskip 1cm

%%%%%%%%%%%%%%%%%%%%%%%%%%%%%%%%%%%%%%%%%%%%%%%%%%%%%%%%%%%%%%%%
%%%%%%%%%%%%%%%%%%%%%%%%%%%%%%%%%%%%%%%%%%%%%%%%%%%%%%%%%%%%%%%%
\setcounter{equation}{0}
\section{Proof of Theorem~\ref{Main}}
\label{section6}
\bigskip

{\bf Proof of Theorem~\ref{Main} :}
Fix a spin structure $(\N,\Phi)$ on $\rho:\D\to \Delta$ given  in Section~\ref{section5}.
Consider the space $\M_{m,0}$ in (\ref{CenterModuli}) where $m$ is a partition of $d=3$.
In this case, by Lemma~\ref{L:CModuli} either $m=(1^3)$ or $m=(3)$. Note that both of them satisfy (\ref{assumption-n}).
Fix $\la\ne 0$ and let $f=(f_1,f_0,f_2)$ be a map in $\M_{m,0}$.
Then (\ref{SM-Fmoduli}) and (\ref{main-parity}) show that  for all $f_\mu\in \CZ_{m,f,\la}$
\begin{equation}\label{MainTask}
p(f_\mu)\ \equiv\ p(f_1)+ p(f_2)\ \ \ \mbox{(mod\ 2)}.
\end{equation}
Lemma~\ref{SHbyRelModuli} and (\ref{DecMr}) show that
\begin{equation}\label{MThmE1}
H_{(3)^k}^{h,p}\ =\ H_{(3)^k,(1^3)^3}^{h,p}\ =\
\frac{1}{(3!)^3}\,\Big(
\sum_{f\in \M_{(1^3),0}}  \sum_{f_\mu\in \Z_{(1^3),f,\la}}\!\!\!\!\! (-1)^{p(f_\mu)} \ +
\sum_{f\in \M_{(3),0}}  \sum_{f_\mu\in \Z_{(3),f,\la}}\!\!\!\!\! (-1)^{p(f_\mu)}\ \Big)
\end{equation}
By (\ref{number-fiber}) and  (\ref{MainTask}), (\ref{MThmE1}) becomes
\begin{equation}\label{MThmE2}
H_{(3)^k}^{h,p}\ =
\sum_{f=(f_1,f_0,f_2)\in \M_{(1^3),0}} \!\!\!\!\!\!\!\frac{(-1)^{p(f_1)+p(f_2)}}{(3!)^3} \ +
\sum_{f=(f_1,f_0,f_2)\in \M_{(3),0}}\!\!\!\!\!\!\! \frac{3^2(-1)^{p(f_1)+p(f_2)}}{(3!)^3}
\end{equation}
It then follows from Lemma~\ref{GluingMap} and (\ref{MThmE2}) that
\begin{align*}\label{MThmE3}
H_{(3)^k}^{h,p}\ &=\
\sum_{(f_1,f_0,f_2)\in \CP_{(1^3)}}\!\!\! \frac{(-1)^{p(f_1)+p(f_2)}}{(3!)^5}\ +
\sum_{(f_1,f_0,f_2)\in \CP_{(3)}} \!\!\!\frac{3^2(-1)^{p(f_1)+p(f_2)}}{(3!)^3}  \notag \\
&=\
\frac{1}{(3!)^3}
\sum_{f_1\in \M_{(1^3)}^1} (-1)^{p(f_1)} \sum_{f_2\in \M_{(1^3)}^2}  (-1)^{p(f_2)} +
\frac{3}{(3!)^2}
\sum_{f_1\in \M_{(3)}^1} (-1)^{p(f_1)} \sum_{f_2\in \M_{(3)}^2} (-1)^{p(f_2)}
\notag \\
&=\
3!\,H_{(3)^{k_1}}^{h_1,p_1}\cdot H_{(3)^{k_2}}^{h_2,p_2}\ +\
3\,H_{(3)^{k_1+1}}^{h_1,p_1}\cdot H_{(3)^{k_2+1}}^{h_2,p_2}
\end{align*}
where the second equality follows from Lemma~\ref{NumM0} and the last from  Lemma~\ref{SHbyRelModuli}.
\qed

\vspace{1cm}

%%%%%%%%%%%%%%%%%%%%%%%%%%%%%%%%%%%%%%%%%%%%%%%%%%%%%%%%%%%%
%%%%%%%%%%%%%%%%%  Section 11  --- Computational Examples     %%%%%%%%%%%%%%%%%%%%%%%%
%%%%%%%%%%%%%%%%%%%%%%%%%%%%%%%%%%%%%%%%%%%%%%%%%%%%%%%%%%%%%%

\setcounter{equation}{0}
\section{Calculation }
\label{section7 }
\bigskip

The aim of this section is to show:

\begin{prop}\label{computation}\ \
${\displaystyle H_{(3)^k}^{h,\pm}\ =\ 3^{2h-2}\big[\,(-1)^k2^{k+h- 1}\pm 1\,\big]}$.
\end{prop}

\begin{proof}
The proof consists of four steps.

\medskip
\non
{\bf Step 1 :}
We first show the following facts which we use in the computation below.

\begin{lemma}\label{known}
$$
\begin{array}{ccc}
\hspace{-2cm}
(a)\ \ H_{(3)^0}^{0,+}=H_3^{0,+}=\tfrac{1}{3!}\ \ \ \  \ \ \  \
&
(b)\ \ H_{(3)^3}^{0,+}=-\tfrac13\ \ \ \ \  \  \ \
&
(c)\ \ H_{(3)^0}^{1,+}=H_3^{1,+}=2
\end{array}
$$
\end{lemma}

\begin{proof}
Consider the dimension zero space $\M^V_\chi(\P^1,3)$ where $V=\emptyset$.
The Euler characteristic $\chi=6$ by (\ref{RHformula}) and hence
the space contains only one map $f:C\to \P^1$ where $C$ is a disjoint union of three rational curves and
$|\Aut(f)|=3!$. This shows (a). Let
$(f,C)$ be a map in the dimension zero space $\M^V_{\chi,(3),(3),(3)}(\P^1,3)$.
Then $C$ is a connected curve of genus one  and
the theta characteristic
$L_f$ on $C$ defined by  (\ref{TwistedTheta})  is
\begin{align*}
L_f\
%&=\ f^*\O(-1)\otimes\O(x_1+x_2+x_3)\\
=\ \O(-2x_1+x_2+x_3)\ =\ \O(x_1-2x_2+x_3)\ =\ \O(x_1+x_2-2x_3)
\end{align*}
where $x_1,x_2$ and $x_3$ are ramification points of $f$. This implies $L_f^3=\O$ and hence
$L_f=\O$ because $L_f^2=L_f^3=\O$. We have $p(f)=1$.
Therefore,
$$
H_{(3)^3}^{0,+}\ =\ -H_{(3)^3}^{0}\ =\ - \tfrac13
$$
where $H_{(3)^3}^0$ denotes the (ordinary) Hurwitz number  which is calculated by using the character formula
(cf.  (0.10) of \cite{OP}).
By Proposition 9.2 of \cite{LP1},
the spin Hurwitz numbers $H_d^{h,p}$ are the dimension zero local invariants of spin curve
that count maps from possibly  disconnected domains.
Let $GW_d^{h,p}$ denote the dimension zero local invariants of spin curve that count maps from connected domains.
Then $H_d^{h,p}$ and $GW_d^{h,p}$ are related as follows:
$$
1+\sum_{d>0} H_d^{h,p}\, t^d\ =\ \exp\big(\sum_{d>0} GW_d^{h,p}\, t^d\big).
$$
Now, (c) follows from: $GW_1^{1,+}=1$, $GW_2^{1,+}=\frac12$ and $GW_3^{1,+}=\tfrac43$ (see Section 10 of \cite{LP1}).
\end{proof}

\medskip
\non
{\bf Step 2 :} In this step, we compute $H_{(3)^k}^{1,-}$.
For a spin curve of genus one with trivial theta characteristic,
it follows from the formula (3.12) of \cite{EOP} that
\begin{equation}\label{EOP}
H_{(3)^k}^{1,-}  \ =\  2^{-k} \left[\big(\,\f_{(3)}(21)\,\big)^k - \big(\f_{(3)}(3)\big)^k \right].
\end{equation}
Here  the so-called {\em central character} $\f_{(3)}$ can be written as
$\f_{(3)}= \frac13\,\p_3 + a_2\p_1^2 + a_1\p_1 + a_0$
for some $a_i\in\Q$ ($0\leq i\leq 2$) and the {\em supersymmetric functions} $\p_1$ and $\p_3$ are defined by
\begin{equation*}\label{ssf}
\p_1(m)  \ =\  d - \tfrac{1}{24}\ \ \ \ \ \ \
\mbox{and}\ \ \ \ \ \ \p_3(m)\  = \ \tsum_j\, m_j^3 - \tfrac{1}{240}
\end{equation*}
where $m=(m_1,\cdots,m_\ell)$ is a partition of $d$.
For $k=0,1$, (\ref{EOP}) shows
\begin{equation}\label{EOP01}
H_{(3)^0}^{1,-}\ =\ 0\ \ \ \ \ \ \ \mbox{and}\ \ \ \ \ \  \
H_{(3)}^{1,-}\ =\ -3.
\end{equation}
Lemma~\ref{known}\,b, (\ref{EOP01}) and the formula (\ref{main1}) give
\begin{equation}\label{EOP2}
H_{(3)^2}^{1,-} \ =\  3\, H_{(3)}^{1,-}\cdot H_{(3)^3}^{0,+} \ =\  3.
\end{equation}
By (\ref{EOP}),  (\ref{EOP01}) and (\ref{EOP2}) we conclude
\begin{equation}\label{character}
\f_{(3)}(21)\ =-\ 4\ \ \ \ \ \mbox{and}\ \ \ \ \ \f_{(3)}(3)\ =\ 2.
\end{equation}
Consequently, by (\ref{EOP}) and (\ref{character}), for $k\geq 0$ we have
\begin{equation}\label{InitialTinv}
H_{(3)^k}^{1,-}\ =\ (-1)^k 2^k - 1.
\end{equation}

\medskip
\non
{\bf Step 3 :} In this step, we compute $H_{(3)^k}^{h,+}$ for $h=0,1$.
For $k\geq 1$, (\ref{EOP01}) and  the formula (\ref{main1})  give
\begin{equation}\label{Tinv-h=0-k1}
H_{(3)^{k-1}}^{1,-} \ =\  3\,H_{(3)}^{1,-}\cdot H_{(3)^k}^{0,+}\ =\ -3^2\,H_{(3)^k}^{0,+}.\ \ \
\end{equation}
Combining Lemma~\ref{known}\,a and (\ref{Tinv-h=0-k1}) yields that for $k\geq 0$
\begin{equation}\label{Tinv-h=0}
H_{(3)^k}^{0,+} \ = \ -\tfrac{1}{3^2} \big(\,(-1)^{k-1}2^{k-1}-1\,\big).
\end{equation}
Lemma~\ref{known}\,c, (\ref{InitialTinv}), (\ref{Tinv-h=0}) and  the formula (\ref{main1}) show
$$
\begin{array}{l}
H_{(3)^0}^{2,+}\ =\ 3!\,H^{1,-}_{(3)^0}\cdot H^{1,-}_{(3)^0}  +
3\,H^{1,-}_{(3)}\cdot H^{1,-}_{(3)}\ =\ 27,_{\displaystyle \phantom{\sum}}\\
%%%%%%%%%%%%%%%%%%%%%%%%%%%%%%%%%%%%%%%%%%%%%%%%%%%%%%%%%%%%%%%%%%%%%%%
H_{(3)}^{2,+}\ =\ 3!\,H^{1,-}_{(3)^0}\cdot H^{1,-}_{(3)}  +
3\,H^{1,-}_{(3)}\cdot H^{1,-}_{(3)^2}\ =\ -27,_{\displaystyle \phantom{\sum}}\\
%%%%%%%%%%%%%%%%%%%%%%%%%%%%%%%%%%%%%%%%%%%%%%%%%%%%%%%%%%%%%%
H_{(3)^0}^{2,+}\ =\
3!\,H^{1,+}_{(3)^0}\cdot H^{1,+}_{(3)^0}  + 3\,H^{1,+}_{(3)}\cdot H^{1,+}_{(3)}\ =\
24 + 3\,H^{1,+}_{(3)}\cdot {H^{1,+}_{(3)}},_{\displaystyle \phantom{\sum}}  \\
%%%%%%%%%%%%%%%%%%%%%%%%%%%%%%%%%%%%%%%%%%%%%%%%%%%%%%%%%%%%%%%%%
%H_{(3)}^{2,+}\ =\ 3\,H^{1,-}_{(3)}\cdot H^{1,-}_{(3)^2}\ =\
%-27_{\displaystyle \phantom{\sum}}^{\displaystyle \phantom{\sum}} \\
%%%%%%%%%%%%%%%%%%%%%%%%%%%%%%%%%%%%%%%%%%%%%%%%%%%%%%%%%%%%%%%%%%%%
H_{(3)}^{2,+}\ =\
3!\,H^{1,+}_{(3)^0}\cdot H^{1,+}_{(3)}  +   3\,H^{1,+}_{(3)}\cdot H^{1,+}_{(3)^2}\ =\
12\,  H^{1,+}_{(3)}  +   3\,H^{1,+}_{(3)}\cdot H^{1,+}_{(3)^2},\\
%%%%%%%%%%%%%%%%%%%%%%%%%%%%%%%%%%%%%%%%%%%%%%%%%%%%%%%%%%%%%%%%%%%%
H^{1,+}_{(3)^2}\ =\
3!\,H^{1,+}_{(3)^0}\cdot H^{0,+}_{(3)^2}  +  3\,H^{1,+}_{(3)}\cdot H^{0,+}_{(3)^3}
\ =\ 4  - {H^{1,+}_{(3)}}.^{\displaystyle \phantom{\sum}}
\end{array}
$$
It follows that $H^{1,+}_{(3)}=-1$.
Consequently, Lemma~\ref{known}\,c, (\ref{Tinv-h=0}) and the formula (\ref{main1}) give
\begin{equation}\label{Tinv-h=1}
H^{1,+}_{(3)^k} \ =\
3!\,H_{(3)^0}^{1,+}\cdot H_{(3)^k}^{0,+} + 3\,H_{(3)}^{1,+}\cdot H_{(3)^{k+1}}^{0,+} \ =\  (-1)^k2^k+1.
\end{equation}

\medskip
\non
{\bf Step 4 :} It remains to compute $H_{(3)^k}^{h,p}$ for $h\geq 2$.
The formula (\ref{main1}) gives
\begin{equation*}
H_{(3)^k}^{h,p}\  =\
3!\,H_{(3)^0}^{h-1,p}\cdot H_{(3)^k}^{1,+} +
3\,H_{(3)}^{h-1,p}\cdot H_{(3)^{k+1}}^{1,+}.
\end{equation*}
From this, we can deduce that for $h\geq 2$
\begin{align}
\left(
\begin{array}{l}
H_{(3)^k}^{h,p} \\ H_{(3)^{k+1}}^{h,p}
\end{array}
\right)
&=
\left(
\begin{array}{ll}
3!\,H_{(3)^k}^{1,+} & 3\,H_{(3)^{k+1}}^{1,+} \\
3!\,H_{(3)^{k+1}}^{1,+} & 3\,H_{(3)^{k+2}}^{1,+}
\end{array}
\right)
\left(
\begin{array}{l}
H_{(3)^0}^{h-1,p} \\ H_{(3)}^{h-1,p}
\end{array}
\right) \notag \\
%%%%%%%%%%%%%%%%%%%%%%%%%%%%%%%%%%%%%%%%%%
&=
\left(
\begin{array}{ll}
3!\,H_{(3)^k}^{1,+} & 3\,H_{(3)^{k+1}}^{1,+} \\
3!\,H_{(3)^{k+1}}^{1,+} & 3\,H_{(3)^{k+2}}^{1,+}
\end{array}
\right)
\left(
\begin{array}{ll}
3!\,H_{(3)^0}^{1,+} & 3\,H_{(3)}^{1,+} \\
3!\,H_{(3)}^{1,+} & 3\,H_{(3)^2}^{1,+}
\end{array}
\right)^{h-2}
\left(
\begin{array}{l}
H_{(3)^0}^{1,p} \\ H_{(3)}^{1,p}
\end{array}
\right)  \label{deg3-final}
\end{align}
Therefore, (\ref{InitialTinv}), (\ref{Tinv-h=1}) and  (\ref{deg3-final}) complete the proof.
\end{proof}

\vskip 1cm

{\small

\vspace{1.3cm}

\noindent {\em  Department of  Mathematics,  University of Central Florida, Orlando, FL 32816}

\medskip

\noindent {\em e-mail:}\ \ {\ttfamily junlee@mail.ucf.edu

}


\medskip
\begin{thebibliography}{}



\bibitem[A]{A} F. Atiyah, {\em Riemann surfaces and spin structures},
Ann. scient. Ec. Norm. Sup. {\bf 4} (1971), 47 -- 62.




\bibitem[ACG]{ACG}
E.~Arbarello, M.~Cornalba, P.~Griffiths, {\em  Geometry of algebraic curves: Volume II},
Springer-Verlag, Berlin, 2011.




%\bibitem[ACV]{ACV} D. Abramovich, A. Corti, and A. Vistoli, {\em Twisted Bundles and Admissible Covers}, Commun. in Algebra. {\bf 31} (2003), 3547-3618.



%\bibitem[AS]{AS} M.F. Atiyah and I.M. Singer, {\em The index of elliptic operators: IV}, Annals of Math. {\bf 93} %(1971), 119-138.


%\bibitem[BF]{BF} K. Behrend and B. Fantechi, {\em The intrinsic normal cone}, Invent. Math. {\bf 128} (1997), no. 1, 45--88.



\bibitem[C]{C} M. Cornalba,  {\em Moduli of curves and theta charateristics},
Lectures  on Riemann Surfaces, 560--589,   World Scientific, Singapore 1989.


\bibitem[EOP]{EOP} A. Eskin, A. Okounkov and R. Pandharipande,
{\em The theta characteristic of a branched covering}, Adv. Math. {\bf 217} (2008), no. 3, 873--888.





%\bibitem{ip0} E. Ionel and T.H.  Parker,
%                   {\em  The  Gromov invariants of Ruan-Tian and Taubes},
%                   Math. Res. Lett. {\bf 4}(1997), 521-532, MR1470424, Zbl 0889.57030.

%\bibitem[FO]{FO} K. Fukaya, K. Ono, {\em Arnold conjecture and Gromov-Witten
%invariant}, Topology {\bf 38} (1999), 933-1048.


%\bibitem[FP]{FP} C. Faber and  R. Pandharipande, {\em Hodge integrals and Gromov-Witten theory}, Invent. Math. %{\bf139} (2000), no. 1, 173--199.

\bibitem[G]{G} S. Gunningham, {\em Spin Hurwitz numbers and topological quantum field theory}, preprint, arXiv:1201.1273.


\bibitem[HM]{HM}  J. Harris and I. Morrison, {\em Moduli of curves}, Springer, New York,  1998.



\bibitem[IP1]{IP1} E. Ionel and T.H.  Parker,
                  {\em Relative Gromov-Witten invariants},
                  Annals of Math.  {\bf 157} (2003), 45-96.


\bibitem[IP2]{IP2} E. Ionel and T.H.  Parker,
                  {\em The yymplectic sum formula for Gromov-Witten invariants},
                  Annals of Math.  {\bf 159} (2004), 935-1025.

%\bibitem[IS1]{IS1} S. Ivashkovich and V. Shevchishin,
%{\em Structure of the moduli space in a neighborhood of a cusp-curve and meromorphic hulls},
%Invent. Math. 136 (1999), no. 3, 571Ð602.

%\bibitem[IS]{IS} S. Ivashkovich and V. Shevchishin,
%{\em Gromov compactness theorem for $J$-complex curves with
%boundary}, Internat. Math. Res. Notices {\bf 2000}, no. 22,
%1167-1206.


\bibitem[KL1]{KL1} Y-H. Kiem and J. Li, {\em Gromov-Witten invariants of
Varieties with holomorphic 2-forms}, preprint, math.AG/0707.2986.



\bibitem[KL2]{KL2} Y-H. Kiem and J. Li,
{\em Low degree GW invariants of spin surfaces}, Pure Appl. Math. Q. {\bf 7} (2011), no. 4, 1449--1476.


%\bibitem[KL3]{KL3} Y-H. Kiem and J. Li,
%{\em Low degree GW invariants of surfaces II}, Science China Math. Vol. 54, No. 8. 1679–1706.








%\bibitem[Lee1]{L}  J. Lee, \textit{Family Gromov-Witten Invariants for
%K\"{a}hler Surfaces},  Duke Math. J. \textbf{ 123} (2004), no 1,
%209--233.


\bibitem[L]{L} J. Lee, {\em Sum formulas for local Gromov-Witten invariants of spin curves}, to appear in
Trans. Amer. Math. Soc.



\bibitem[LP1]{LP1} J. Lee and T.H. Parker,
{\em A structure Theorem for the Gromov-Witten invariants of
K\"{a}hler surfaces}, J. Differential Geom. {\bf 77} (2007), no. 3,
483--513.

\bibitem[LP2]{LP2} J. Lee and T.H. Parker,
{\em A recursion formula for spin Hurwitz numbers}, preprint.


%\bibitem[LT1]{LT1} J. Li and G. Tian, {\em Virtual moduli cycles and
%Gromov-Witten invariants of general symplectic manifolds}, Topics in
%symplectic $4$-manifolds (Irvine, CA, 1996), 47--83, First Int.
%Press Lect. Ser., I, Internat. Press, Cambridge, MA, 1998.

%\bibitem[LT2]{LT2}
%J. Li and G. Tian, {\em Virtual moduli cycles and Gromov-Witten
%invariants of algebraic varieties}, J. Amer. Math. Soc. {\bf 11}
%(1998), no. 1, 119--174.


%\bibitem[LT3]{LT3}
%J. Li and G. Tian, {\em Comparison of algebraic and symplectic
%Gromov-Witten invariants}, Asian J. Math. {\bf 3} (1999), no. 3,
%689--728.


\bibitem[M]{M} D. Mumford, {\em Theta characteristics of an algebraic curves}, Ann. scient. Ec. Norm. Sup. {\bf 4}
(1971), 181 -- 192.





\bibitem[MP]{MP} D. Maulik and R. Pandharipande, {\em New calculations in Gromov-Witten theory},
Pure Appl. Math. Q. {\bf 4} (2008), no. 2, part 1, 469--500.


%\bibitem[MS]{MS} D. McDuff and D. Salamon, {\em $J$-holomorphic Curves and Symplectic Topology}, AMS Colloquium %Publications Vol. 52, AMS, Providence, RI., 2004.



\bibitem[OP]{OP}  A. Okounkov, R. Pandharipande,
{\em Gromov-Witten theory, Hurwitz theory, and completed cycles}, Ann. of Math. (2) {\bf 163} (2006), no. 2, 517--560.

%\bibitem[OP2]{OP2}  A. Okounkov, R. Pandharipande,
%{\em Virasoro constraints for target curves}, Invent. Math. {\bf 163} (2006), no. 1, 47--108.





%\bibitem[RT1]{RT1} Y. Ruan and G. Tian,
%{\em A mathematical theory of quantum cohomology}, J. Differential
%Geom. {\bf 42} (1995), 259-367.



%\bibitem[RT2]{RT2}  Y. Ruan and G. Tian, {\em Higher genus symplectic invariants
%and sigma models coupled with gravity},  Invent. Math. {\bf 130} (1997), no. 3, 455--516.

%\bibitem[S]{S} B. Siebert, {\em Gromov-Witten invariants for general symplectic manifolds}, preprint, arXive:9608005.

%\bibitem[T]{T} C.H. Taubes,  {\em Self-dual connections  on 4-manifolds with indefinite intersection form},
%J. Diff. Geometry, {\bf 19} (1984), 517--560.

%\bibitem[Th]{Th}  R. Thom, {\em Quelques propriétés globales des variétés différentiables},
%Comment. Math. Helv. {\bf 28}, (1954), 17--86

%\bibitem[Z]{Z} A. Zinger, {\em A Comparison Theorem for Gromov-Witten Invariants in the Symplectic Category}, %preprint, arXive:0807.0805

\end{thebibliography}
\end{document}